\documentclass[11pt,a4paper,reqno,dvipsnames]{amsart}
\allowdisplaybreaks[4]

\usepackage{amsfonts}
\usepackage{amsmath}
\usepackage{amssymb}
\usepackage{amsthm}
\usepackage{amsxtra}
\usepackage{bbm}
\usepackage{braket}
\usepackage{comment}
\usepackage{extarrows}
\usepackage{graphicx}
\usepackage{latexsym}
\usepackage{mathrsfs}
\usepackage{yhmath}
\usepackage{nicematrix}
\usepackage{mathtools}

\usepackage{float}
\usepackage{verbatim}
\usepackage{xcolor}
\usepackage[normalem]{ulem}
\usepackage{caption,subcaption}

\usepackage[pdfborder={0 0 0}]{hyperref} 
\hypersetup{breaklinks, raiselinks}

\usepackage{bm}
\usepackage{tikz-cd}

\usepackage[initials,lite]{amsrefs} 
\AtBeginDocument{%
\def\MR#1{}
}

\usepackage{cleveref}
\Crefname{Lemma}{Lemma}{Lemmas}
\Crefname{Theorem}{Theorem}{Theorems}

\usepackage{geometry}
\theoremstyle{plain}
\newtheorem{Theorem}{Theorem}[section]
\newtheorem{Lemma}[Theorem]{Lemma}
\newtheorem{Corollary}[Theorem]{Corollary}
\newtheorem{Proposition}[Theorem]{Proposition}

\theoremstyle{definition}

\newtheorem{Assumptions and Discussion}[Theorem]{Assumptions and Discussion}

\newtheorem{Definition}[Theorem]{Definition}

\newtheorem{Remark}[Theorem]{Remark}

\theoremstyle{remark}

\newtheorem{Setting}[Theorem]{Setting}

\newtheorem*{acknowledgment*}{Acknowledgment}

\usepackage[shortlabels]{enumitem}
\SetEnumerateShortLabel{a}{\textup{(\alph*)}}
\SetEnumerateShortLabel{A}{\textup{(\Alph*)}}
\SetEnumerateShortLabel{1}{\textup{(\arabic*)}}
\SetEnumerateShortLabel{i}{\textup{(\roman*)}}
\SetEnumerateShortLabel{I}{\textup{(\Roman*)}}

\def\lex{\operatorname{lex}}

\def\bar#1{\overline{#1}}

\def\depth{\operatorname{depth}}

\def\dim{\operatorname{dim}}

\def\e{\operatorname{e}}

\def\Hilb{\operatorname{Hilb}}

\def\ini{\operatorname{in}} 

\def\KK{{\mathbb K}}

\def\lex{{\operatorname{lex}}}

\def\reg{\operatorname{reg}}

\def\cd{\operatorname{cd}}

\def\depth{\operatorname{depth}}

\def\ZZ{{\mathbb Z}}

\newcommand\bdX{\bm{X}}

\newcommand\calC{\mathcal{C}}

\newcommand\calJ{\mathcal{J}}

\newcommand\fraka{\mathfrak{a}}

\newcommand\frakb{\mathfrak{b}}

\newcommand\frakm{\mathfrak{m}}

\begin{document}
\title{Generalized binomial edge ideals of complete $r$-partite graphs}

\author{Yi-Huang Shen}
\address{CAS Wu Wen-Tsun Key Laboratory of Mathematics, School of Mathematical Sciences, University of Science and Technology of China, Hefei, Anhui, 230026, P.R.~China}
\email{yhshen@ustc.edu.cn}

\author{Guangjun Zhu$^{\ast}$}
\address{School of Mathematical Sciences, Soochow University, Suzhou, Jiangsu, 215006, P.R.~China}
\email{zhuguangjun@suda.edu.cn}

\thanks{$^{\ast}$ Corresponding author}


\begin{abstract}
    This paper analyzes the cohomological dimension of the generalized binomial edge ideal $\calJ_{K_m,G}$ for a complete $r$-partite graph $G$. Additionally, the Krull dimension, the depth, the Castelnuovo--Mumford regularity, the Hilbert series, and the multiplicity of its quotient ring are explicitly determined. 
\end{abstract}  

\maketitle

\section{Introduction}
When $n$ is a positive integer, $[n]$ represents the set $\{1,2,\dots,n-1,n\}$ according to convention. Suppose that $m$ and $n$ are positive integers.
Let $S=\KK[\bdX]\coloneqq \KK[x_{ij}:(i,j)\in[m]\times [n]]$ be the polynomial ring in $mn$ variables over a field $\KK$.
In their paper \cite{MR3290687}, Ene et al.~introduced the {binomial edge ideal of a pair of graphs}. Specifically, let $G_1$ and $G_2$ be two simple graphs on the vertex sets $[m]$ and $[n]$ with the  edge sets $E(G_1)$ and $E(G_2)$, respectively. Given $e=\{i, j\}\in E(G_1)$ and $f=\{k, l\}\in E(G_2)$, where $i<j$ and $k<l$, we can assign a $2$-minor $p_{(e,f)}=[i,j\,|\,k,l]\coloneqq x_{ik}x_{jl}-x_{il}x_{jk}$ to the pair $(e, f)$. The \emph{binomial edge ideal of the pair $(G_1, G_2)$} is defined as the ideal
\[
    \calJ_{G_1,G_2}\coloneqq (p_{(e,f)} : e \in E(G_1), f \in E(G_2))
\]
in $S$.

Let $G$ be a finite simple graph with the vertex set $[n]$. The \emph{generalized binomial edge ideal} of $G$ is the binomial edge ideal of the pair $(K_m, G)$, where $K_m$ is a complete graph with $m$ vertices.
Generalized binomial edge ideals have numerous applications in algebraic statistics for studying conditional independence, as noted in \cite{MR3011436}. These ideals are a generalization of the classical binomial edge ideals studied in \cites{MR2669070, MR2782571}. It is not surprising that certain properties of the classical binomial edge ideal $J_G$ can be naturally extended to $\mathcal{J}_{K_m,G}$. For instance, $\mathcal{J}_{K_m,G}$ is a radical ideal, and the subsets of $G$ with the cut-point property determine the minimal primes of $\calJ_{K_m,G}$, as proven in \cite{MR3011436}.

However, for $m\ge 3$, the generalized binomial edge ideal $\calJ_{K_m,G}$ cannot be generated by maximal minors of the matrix $\bdX$. This gives rise to new phenomena arise. For example, Chaudhry and Irfan in \cite{MR4233116}  proved  that, for a block graph $G$, the ideal $\calJ_{K_m,G}$ is Cohen--Macaulay if and only if $\calJ_{K_m,G}$ is unmixed, if and only if $G$ is a complete graph.  Kumar \cite{MR4033090} demonstrated 
that the upper bound of the regularity $\reg(S/\calJ_{K_m,G})$ is $n-1$, which is achieved if $m\ge n$.

Building on their work, we explored the algebraic properties of generalized binomial edge ideals for specific graphs. In \cite{arXiv:2305.05365}, we provided precise formulas for the Krull dimension, depth and regularity of $S/\calJ_{K_m,G}$, where $G$ is either a bipartite graph $F_p$ or a $k$-fan graph $F_k^{W}(K_n)$. These graphs serve as the foundation for constructing Cohen--Macaulay binomial edge ideals of bipartite graphs. In a recent study \cite{arXiv:2310.20235}, we analyzed the regularity and depth of the  powers of  the generalized binomial edge ideal $\calJ_{K_m,P_n}$ of a path graph $P_n$ using the Sagbi basis theory.
Furthermore, we showed that the symbolic powers and ordinary powers of $\calJ_{K_m,P_n}$ are identical. 

This paper focuses on the generalized binomial edge ideals of complete $r$-partite graphs.
In their work \cite{MR3195706}, Schenzel and Zafar  computed the Krull dimension, depth, regularity, Hilbert function and multiplicity of the binomial edge ideal of a complete bipartite graph. Wang and Tang also studied the regularity and  depth of the powers of the binomial edge ideal of a complete bipartite graph in their papers \cites{MR4563443, MR4632858}.  Meanwhile, Ohtani \cite{MR3169495} demonstrated that the symbolic powers and ordinary powers of $J_{G}$ coincide when $G$ is a complete $r$-partite graph. 
However, there is limited understanding of the generalized binomial edge ideal of a complete $r$-partite graph.

The cohomological dimension of binomial edge ideals is also an intriguing topic. In their work \cite{MR1082012}, Bruns and Schw\"anzl  computed the cohomological dimension of $\calJ_{K_m,K_n}$. Katsabekis \cite{arXiv:2207.02256} provided a lower bound on the cohomological dimension of $\calJ_{K_m,G}$ for any non-complete graph $G$, based on its vertex connectivity. Particularly, he obtained interesting results for path graphs and star graphs. What can be said about the cohomological dimension of the generalized binomial edge ideal of complete $r$-partite graphs?

The article is organized as follows. Section \ref{sec:prelim} presents essential definitions and terminology that will be necessary later.
In Section \ref{sec:partite},
we begin by analyzing the primary decomposition of $\calJ_{K_m,G}$, where $G$ is  a complete $r$-partite graph. This enables us to quickly obtain the Krull dimension and the depth of $S/\calJ_{K_m,G}$. As a corollary, we generalize a recent result of Katsabekis on the cohomological dimension.
We then compute the Castelnuovo--Mumford regularity, Hilbert series, and  multiplicity of $S/\calJ_{K_m,G}$. It is important to note that Schenzel and Zafar computed the regularity of binomial edge ideals of complete bipartite graphs as a by-product of a long and meticulous analysis of the modules of deficiencies. In contrast, we present here an exact formula for the regularity of the generalized binomial edge ideal of a complete $r$-partite graph, using the comparison-by-induced-subgraph strategy, which significantly shortens the proof.
Finally, we conclude this paper by demonstrating that the binomial edge ideal of a complete $r$-partite graph is of K\"onig type, following the definition of Herzog et al.~in \cite{MR4358668}. This generalizes a recent result of Williams in \cite{arXiv:2310.14410} for complete bipartite graphs.

\section{Preliminaries}
\label{sec:prelim}

This section provides the necessary definitions and basic facts for the paper. For further details, please refer to \cites{MR1613627, MR1251956, MR3838370}.

\subsection{Notions of simple graphs}
To begin, let us review some fundamental concepts regarding graphs. Let $G$ be a simple graph with the vertex set $V(G)$ and the edge set $E(G)$. For any subset $A$ of $V(G)$, the \emph{induced subgraph} of $G$ on the vertex set $A$, denoted by $G|_{A}$, satisfies that $V(G|_A)=A$ and for any $i,j \in A$, $\{i,j\} \in E(G|_{A})$ if and only if $\{i,j\}\in E(G)$. At the same time, the induced subgraph of $G$ on the set $V(G)\setminus A$ will be denoted by $G\setminus A$. In particular, if $A=\{v\}$  then we will write $G\setminus v$ instead of $G\setminus \{v\}$ for simplicity.

Given positive integers $n_1$ and $n_2$, where $n_1 \le  n_2$, the set $[n_1,n_2]$ is defined as $\{n_1,n_1+1,\dots,n_2-1,n_2\}$. If $n_1=1$, we will use the simplified notation $[n_2]$.

Let $r$ be an integer greater than or equal to $2$. A simple graph $G$ is \emph{$r$-partite} if its vertex set $V(G)$ can be partitioned into $r$ pairwise disjoint subsets, such that no two vertices in the same subset are adjacent in $G$.  A \emph{complete $r$-partite} graph with $\sum_{i=1}^{r}n_i$  vertices, denoted by $K_{n_1,n_2,\ldots,n_r}$, is a graph with the vertex set $V(K_{n_1,n_2,\ldots,n_r})=[n_1]\sqcup [n_1+1,n_1+n_2]\sqcup\cdots \sqcup  [1+\sum_{i=1}^{r-1}n_i,\sum_{i=1}^{r}n_i]$ and the edge set $E(K_{n_1,n_2,\ldots,n_r})=\{\{u, v\}: u\in V_i, v\in V_j  \text{\ for different\ } i,j \}$.  A graph with $r=2$ is called a \emph{complete bipartite} graph. In contrast, a graph with $n_1=n_2=\cdots=n_r=1$ is known as a \emph{complete graph}, conventionally denoted by $K_r$.

A \emph{walk} $W$ of length $n-1$ in a graph $G$ is a sequence of vertices $w_1$ through $w_{n}$, where each consecutive pair of vertices $\{w_i,w_{i+1}\}$ is connected by an edge in $G$.
A \emph{path} is a walk where all vertices are distinct. To simplify notation, a path of length $n-1$ is denoted by $P_n$.

\subsection{Primary decompositions}

Rauh \cite{MR3011436} has already studied the primary decomposition of generalized binomial edge ideals. 
Specifically, let $m,n\ge 2$ be two positive integers and $G$ be a simple graph with the vertex set $[n]$. For each $T\subseteq [n]$, let $\bar{T}=V(G)\setminus T$, and $G_1,\ldots, G_{c(T)}$ be the connected components of $G|_{\bar{T}}$. In addition, let $\widetilde{G_i}$ be the complete graph on the set $V(G_i)$ for $i=1,\dots,c(T)$. Then, we have the ideal 
\[
    P_T(K_m, G)\coloneqq (x_{ij}: (i,j)\in [m]\times T)+\calJ_{K_m,\widetilde{G}_1}+\cdots+\calJ_{K_m,\widetilde{G}_{c(T)}}
\]
in the polynomial ring $S=\KK[x_{ij}: (i,j)\in[m]\times [n]]$. Note that $P_T(K_m, G)$ is a prime ideal containing $\calJ_{K_m,G}$.

To describe the minimal primes of $\calJ_{K_m,G}$,
we need to impose restrictions on $T$. We define a
\emph{cut point} of $G$ as a vertex $v\in V(G)$
such that $c(G)<c(G\setminus v)$, where $c(G)$
denotes the number of connected components of $G$.
Let $T$ be a subset of $V(G)$, and let $c(T)$
denote the number of connected components of
$G\setminus T$. If $v$ is a cut point of the
induced subgraph $G\setminus (T\setminus \{v\})$
for any $v\in T$, then we say that $T$ has the
\emph{cut point property}.  We define $\calC(G)$ as
the set of all subsets of $V(G)$ that have the cut
point property,
and we let $\overline{\calC}(G)\coloneqq
\calC(G)\setminus \{\emptyset\}$.

Now, we can describe the minimal prime ideals of $\calJ_{K_m,G}$.
\begin{Lemma}
    [{\cite[Theorem 7]{MR3011436}}]
    \label{lem:decompo}
    Let $H$ be a finite simple graph. Then, the generalized binomial edge ideal $\calJ_{K_m,H}$ is reduced and
    \[
        \calJ_{K_m,H}
        =\bigcap_{T\in  \mathcal{C}(H)}P_T(K_m, H)
    \]
    is a minimal primary decomposition of this ideal.
\end{Lemma}

\subsection{Depth and regularity}

In the sequel, let $\frakm$ be the unique graded maximal ideal of $S$ with respect to the standard grading. 
The local cohomology modules of a finitely generated graded $S$-module $M$ with respect to $\frakm$ are denoted by $H_{\frakm}^i(M)$ for $i\in \ZZ$.

\begin{Definition}
    Let $M$ be a finitely generated graded  $S$-module.
    \begin{enumerate}[a]
        \item For each $i$ from $0$ to $\dim(M)$, the \emph{$a_i$-invariant} of $M$ is defined as 
            \[
                a_i(M) \coloneqq \max\{t : (H_{\frakm}^i(M))_t \ne 0\},
            \]
            with the convention that $\max \emptyset = -\infty$.
        \item The \emph{Castelnuovo--Mumford regularity} of $M$ is defined as 
            \[
                \reg(M) \coloneqq \max\{a_i(M) + i:0\le i\le \dim(M)\}.
            \]
    \end{enumerate}  
\end{Definition}

\begin{Remark}
    \label{lem:basic-facts}
    For a finitely generated graded $S$-module $M$,
    with $M\ne 0$, the local cohomology module
    $H_{\frakm}^i(M)=0$ when $i<\depth(M)$ or
    $i>\dim(M)$. However,
    $H_{\frakm}^{\depth(M)}(M)\ne 0$ and
    $H_{\frakm}^{\dim(M)}(M)\ne 0$
    \textup{(}\cite[Theorem
    3.5.7]{MR1251956}\textup{)}. In particular, if
    $M$ is Cohen--Macaulay, then $\reg(M)=a_d(M)+d$
    for $d=\dim(M)$.
\end{Remark}

The lemma below is frequently utilized to calculate a module's depth and regularity.

\begin{Lemma}
    [{\cite[Lemmas 2.1, 3.1]{MR2643966}}]
    \label{depthlemma}
    Let $0\rightarrow M \rightarrow N \rightarrow P \rightarrow 0$ be a short exact sequence of finitely generated graded $S$-modules.  
    \begin{enumerate}[a]
        \item \label{depthlemma-a} One has $\depth(M)\ge \min \{\depth(N), \depth(P)+1\}$. The equality holds if $\depth(N)\neq\depth(P)$.
        \item \label{depthlemma-c} One has $\reg(M)\le\max\{\reg(N),\reg(P)+1\}$. The equality holds if $\reg(N)\neq\reg(P)$.
    \end{enumerate}
\end{Lemma}

The following two lemmas are highly valuable for analyzing the regularity of generalized binomial ideals.

\begin{Lemma}
    [{\cite[Proposition 8]{MR3040610}}]
    \label{lem:induced-graph}
    For a simple graph $G$ and its induced subgraph $H$, $\reg(\mathcal{J}_{K_m,H})$ is less than or equal to $\reg(\mathcal{J}_{K_m,G})$.
\end{Lemma}

\begin{Lemma}
    [{\cite[Theorems 3.6 and 3.7]{MR4033090}}]
    \label{lem:reg-min-equal}
    For a connected graph $G$ on the vertex set $[n]$ and $m \ge 2$, $\reg(S/\calJ_{K_m,G})$ is less than or equal to $n-1$. If $m \ge n \ge 2$, then $\reg(S/\calJ_{K_m,G})$ equals $n-1$.
\end{Lemma}

\section{Generalized binomial edge ideals of a complete $r$-partite graph}
\label{sec:partite}

This section analyzes the generalized binomial edge ideal $\calJ_{K_m,G}$ associated with a complete $r$-partite graph $G$ where $r\ge 2$.  The primary decomposition of $\calJ_{K_m,G}$ is examined to determine the Krull dimension and the depth of $S/\calJ_{K_m,G}$. A recent result of Katsabekis on the cohomological dimension is generalized as a corollary. Additionally, the Castelnuovo--Mumford regularity, the Hilbert series, and the multiplicity of the quotient ring $S/\calJ_{K_m,G}$ are computed.

Assuming $G$ is a complete $r$-partite graph, its vertices can be rearranged to meet the following conditions: 

\begin{Setting}
    \label{r_setting} 
    Let $m, r, n_1, n_2,\dots,n_r$ be positive integers such that $m\ge 2$, $r\ge 2$, $n_r\ge 2$, and $1=n_1= n_2=\cdots=n_{s-1}<n_{s}\le \cdots\le n_r$ for some $s\in [r]$. Suppose that $G$ is a complete $r$-partite graph with  a partition $V_1\sqcup \cdots \sqcup V_r$, where $V_k \coloneqq \left[1+\sum\limits_{j=1}^{k-1}n_j,\sum\limits_{j=1}^{k}n_j\right]$ for each $k\in [r]$. At the same time, let $\widetilde{G}$ be the complete graph on $[n]$ with $n\coloneqq \sum_{j=1}^{r}n_j$.
    As usual, let $\calJ_{K_m,G}$ and $\calJ_{K_m,\widetilde{G}}$ be the generalized binomial edge ideals of the graphs $G$ and $\widetilde{G}$, respectively, in the polynomial ring $S=\KK[x_{ij}:i\in[m],j\in[n]]$. Furthermore, let
    \[
        A_i \coloneqq 
        \begin{cases}
            S, &\text{if $n_i=1$,} \\
            (x_{ij}: i\in [m], j\in T_i), & \text{if $n_i\ge 2$,}
        \end{cases}
    \]
    where $T_i\coloneqq \bigsqcup_{j\ne i}V_j$. 
\end{Setting}

\begin{Remark}
    If $n_1=\cdots=n_r=1$, then $G$ is a complete graph $K_r$. In this scenario, $S/\calJ_{K_m,G}$ is Cohen--Macaulay with $\dim(S/\calJ_{K_m,G})=m+r-1$ and $\reg(S/\calJ_{K_m,G})=\min\{m-1,r-1\}$, as proven by \Cref{lem:generic-Cohen--Macaulay,reg-m-n} below.
\end{Remark}

\begin{Lemma}
    [{\cite[Corollary 4]{MR266912}}]
    \label{lem:generic-Cohen--Macaulay}    
    Let $\bdX=(x_{ij})_{1\le i\le m,1\le j\le n}$ be a generic matrix over $\KK$. Then the ideal $I_t(\bdX)$ in $S$, generated by all $t\times t$ minors of $\bdX$, is perfect of height $(m-t+1)(n-t+1)$. In particular, $\depth(\KK[\bdX]/I_t(\bdX))=(m+n-1)(t-1)$.
\end{Lemma}

\begin{Lemma}
    [{\cite[Proposition 3.3]{MR4033090} or \cite[Theorem 5.1]{MR3764061}}]
    \label{reg-m-n}
    For the complete graph $K_n$ with $n$ vertices, we have $\reg(S/\calJ_{K_m,K_n})=\min\{m-1,n-1\}$.
\end{Lemma}

Now, we will discuss the primary decomposition and the dimension of $\calJ_{K_m,G}$.

\begin{Proposition}
    \label{r_dimension}
    Under \Cref{r_setting}, the minimal primary decomposition is
    \[
        \mathcal{J}_{K_m,G}=\mathcal{J}_{K_m,\widetilde{G}}\cap A_{s}\cap A_{s+1} \cap \cdots\cap A_r,
    \]
    and the dimension is
    \[
        \dim(S/\mathcal{J}_{K_m,G})=\max\{m+n-1, mn_r\}.
    \]
\end{Proposition}

\begin{proof} 
    Notice that $\dim(S/\mathcal{J}_{K_m,\widetilde{G}})=m+n-1$ by \Cref{lem:generic-Cohen--Macaulay} and $\dim(S/A_i)=mn_i$ for each $i\in [r]$. Thus, by \Cref{lem:decompo}, it suffices to show that
    \begin{equation}
        \calC(G)=\Set{\emptyset, T_s, T_{s+1},\dots, T_r}.
        \label{eqn:r_dim_cut_set}
    \end{equation}
    Let $G_i$ be the empty graph on the vertex set $V_i$ for $i\in [r]$. Then, $G$ is the join graph $G_1 * \cdots * G_r$ in the sense of \cite[Section 4]{MR3395714}.
    Therefore, the equality of \eqref{eqn:r_dim_cut_set} follows from \cite[Propositions 4.1, 4.5 and 4.14]{MR3395714}. \qedhere
\end{proof}

Using the primary decomposition of $\calJ_{K_m,G}$, we can calculate the depth of $S/\calJ_{K_m,G}$.

\begin{Theorem}
    \label{thm:depth_r_partite}
    Under \Cref{r_setting}, we have $\depth(S/\mathcal{J}_{K_m,{G}})=m+n_s$.
\end{Theorem}

\begin{proof}
    To utilize the primary decomposition $\calJ_{K_m,G}=\calJ_{K_m,\widetilde{G}}\cap A_{s}\cap \cdots \cap A_r$ from \Cref{r_dimension}, we introduce $\fraka_{s-1}\coloneqq \calJ_{K_m,\widetilde{G}}$ and $\fraka_k\coloneqq \calJ_{K_m,\widetilde{G}}\cap A_{s}\cap \cdots \cap A_k$ for $s\le k\le r$.
    With these preparations,
    for each $k\in [s,r]$, we have
    \[
        0\to S/(\fraka_{s-1}\cap A_k) \to S/\fraka_{s-1} \oplus S/A_k \to S/(\fraka_{s-1}+A_k) \to 0.
    \]
    Note that $\depth(S/\fraka_{s-1})=m+n-1$, $\depth(S/A_k)=mn_k$ and $\depth(S/(\fraka_{s-1}+A_k))=m+n_k-1$. Since $m\ge 2$,  $r\ge 2$ and $n_k\ge 2$, we have $\min\{m+n-1,mn_k\}>m+n_k-1$. So $\depth(S/(\fraka_{s-1}\cap A_k))=m+n_k$ by \Cref{depthlemma}\ref{depthlemma-a}.

    Next, we prove by induction on $k\in [s,r]$ that $\depth(S/\fraka_k)=m+n_{s}$. The base case when $k=s$ is clear from the above calculation, since $\fraka_s=\fraka_{s-1}\cap A_s$. Now, suppose that $k\ge s+1$ and we consider the  short exact sequence
    \[
        0\to S/\fraka_{k} \to S/\fraka_{k-1}\oplus S/(\fraka_{s-1}\cap A_k) \to S/(\fraka_{k-1}+\fraka_{s-1}\cap A_k)\to 0.
    \]
    By induction, we have $\depth(S/\fraka_{k-1})=m+n_{s}$. By the calculation in the previous part, we have $\depth(S/(\fraka_{s-1}\cap A_k))=m+n_k$. Let $A_k^{\complement}\coloneqq (x_{ij}:i\in [m], j\in V_k)\subset S$. Then, we have
    \[
        \fraka_{s-1} \supseteq \fraka_{k-1}+\fraka_{s-1}\cap A_k\supseteq \fraka_{s-1}\cap A_k^{\complement}+\fraka_{s-1}\cap A_k \supseteq \fraka_{s-1}.
    \]
    So  $\fraka_{s-1}=\fraka_{k-1}+\fraka_{s-1}\cap A_k$ and $\depth(S/(\fraka_{k-1}+ \fraka_{s-1}\cap A_k))=\depth(S/\fraka_{s-1})=m+n-1$. Since $k\ge s+1$, we have
    \[
        \min\{m+n_{s},m+n_k\}=m+n_{s}<m+n_s+n_{s+1}-1\le m+n-1.
    \]
    So  we obtain $\depth(S/\fraka_k)=m+n_{s}$ by \Cref{depthlemma}\ref{depthlemma-a}.

    In particular, we can take $k=r$ and obtain that 
    \[
        \depth(S/\calJ_{K_m,G})=\depth(S/\fraka_r)=m+n_{s}. \qedhere
    \]
\end{proof}

As a corollary to \Cref{thm:depth_r_partite}, we will consider the cohomological dimension of $\calJ_{K_m,G}$. Recall that in a Noetherian ring $R$, the \emph{cohomological dimension} $\cd_R(I)$ of an ideal $I$ is defined as follows:
\[
    \cd_R(I)\coloneqq \max\{i\in \ZZ: H_I^{i}(R)\ne 0\}.
\]
Here, $H_I^{i}(R)$ represents the $i$-th local
cohomology module of $R$ with support in $I$. It is
worth noting that the star graph is a special
complete bipartite graph, and the following result
is a generalization of \cite[Proposition
2.12]{arXiv:2207.02256}. 

\begin{Corollary}
    Under \Cref{r_setting}, we have the following results:
    \begin{enumerate}[a]
        \item  If the characteristic of the  field  $\KK$ is positive, then 
            \[
                \cd(\calJ_{K_m,G})= mn-m-n_s.
            \]
        \item  If the characteristic of the  field  $\KK$ is zero, then 
            \[
                mn-m-n_s\le \cd(\calJ_{K_m,G})\le mn-3.
            \]
    \end{enumerate}
\end{Corollary}
\begin{proof}
    Based on \cite[Theorem 2]{MR3011436},
    $\calJ_{K_m,G}$ has a squarefree initial ideal
    with respect to some monomial order on $S$.
    Therefore, by applying
    \Cref{thm:depth_r_partite} and
    \cite[Proposition 3.6]{MR4132955}, we can infer
    that $\cd(\calJ_{K_m,G})\ge mn-m-n_s$, with
    equality when the characteristic of $\KK$ is
    positive.  Additionally, since
    $\depth(R/\calJ_{K_m,G})=m+n_s\ge 3$, as per
    \Cref{thm:depth_r_partite}, we can conclude
    that $\cd(\calJ_{K_m,G})\le mn-3$, according to
    \cite[Theorem 3.5]{MR3078644}. 
\end{proof}

Next, we will compute the regularity of $S/\calJ_{K_m,G}$ under \Cref{r_setting}. Before proceeding, let us examine a special case.

\begin{Lemma}
    \label{star}
    For $m\ge 2$, we have $\reg(S/\calJ_{K_m,K_{1,m}})=m$.
\end{Lemma}
\begin{proof}
    We will treat $K_{1,m}$ here as the complete $r$-partite graph $G$ in \Cref{r_setting}, by taking $r=2$, $n_1=1$, and $n_2=m$. When $m=2$, the result can be derived from \cite[Theorem 4.1]{MR3195706}. In the following, we will assume that $m\ge 3$. It follows from \Cref{r_dimension} that there exists a short exact sequence
    \[
        0\rightarrow  S/\calJ_{K_m,G} \rightarrow  S/\calJ_{K_m,\widetilde{G}}\oplus  S/A_2\rightarrow S/(\calJ_{K_m,\widetilde{G}}+A_2)\rightarrow 0.
    \]
    Note that  $S/\calJ_{K_m,\widetilde{G}}$, $S/A_2$ and $S/(\calJ_{K_m,\widetilde{G}}+A_2)$ are Cohen--Macaulay of dimensions $2m$, $m^2$ and $2m-1$, respectively, by \Cref{lem:generic-Cohen--Macaulay}. 
    Since $m\ge 3$, we have $m^2>2m$.
    Applying the local cohomology functor with respect to $\frakm$ to the above exact sequence, we obtain from \Cref{lem:basic-facts} that 
    \begin{gather}
        0 \to H_{\frakm}^{2m-1}(S/(\calJ_{K_m,\widetilde{G}}+A_2)) \to H_{\frakm}^{2m}(S/\calJ_{K_m,G}) \to H_{\frakm}^{2m}(S/\calJ_{K_m,\widetilde{G}})\to 0, \label{eqn:123-1}\\
        H_{\frakm}^{m^2}(S/\calJ_{K_m,G})\cong H_{\frakm}^{m^2}(S/A_2), \label{eqn:123-2}
        \\
        \intertext{and}
        H_{\frakm}^i(S/\calJ_{K_m,G})=0 \qquad \text{for all $i\notin \{2m,m^2\}$.} \label{eqn:123-3}
    \end{gather}
    
    Note that
    \begin{align*}
        \reg(S/\calJ_{K_m,G})&=\max\{a_i(S/\calJ_{K_m,G})+i: \depth(S/\calJ_{K_m,G})\le i\le \dim(S/\calJ_{K_m,G})\}\\
        &\xlongequal{\eqref{eqn:123-3}} 
        \max\{a_{2m}(S/\calJ_{K_m,G})+2m, a_{m^2}(S/\calJ_{K_m,G})+m^2\}.
    \end{align*}
    Since $S/A_2$ is a polynomial ring of dimension $m^2$ with regularity  $0$, it follows from \eqref{eqn:123-2} that 
    \[
        a_{m^2}(S/\calJ_{K_m,G})=-m^2. 
    \]
    Meanwhile, $S/\calJ_{K_m,\widetilde{G}}$ is Cohen--Macaulay of dimension $2m$ and $\reg(S/\calJ_{K_m,\widetilde{G}})=m-1$ by \Cref{reg-m-n}. So $a_{2m}(S/\calJ_{K_m,\widetilde{G}})=(m-1)-2m=-m-1$. Similarly, we get $a_{2m-1}(S/(\calJ_{K_m,\widetilde{G}}+A_2))=(m-1)-(2m-1)= -m$. Since the short exact sequence \eqref{eqn:123-1} is  graded, we get 
    \begin{align*}
        a_{2m}(S/\calJ_{K_m,G})&=
        \max\{-m-1,-m\}
        =-m.
    \end{align*}
    Consequently, $\reg(S/\calJ_{K_m,G})=m$.
\end{proof}

\begin{Theorem}
    \label{thm:reg_r_partite}
    Under \Cref{r_setting}, we have
    \[
        \reg(S/\mathcal{J}_{K_m,{G}})=\begin{cases}
            n-1, & \text{if $m\ge n$,}\\
            m-1, &\text{if $n>m>n_r$,} \\
            m, &\text{if $n_r\ge m\ge 2$.} 
        \end{cases}
    \]
\end{Theorem}

\begin{proof}
    By \Cref{r_dimension}, we have
    \[
        \calJ_{K_m,G}= A_s \cap \cdots \cap A_r \cap \calJ_{K_m,\widetilde{G}}.
    \]
    Let $\frakb_k\coloneqq A_k \cap \cdots \cap A_r \cap \calJ_{K_m,\widetilde{G}}$ for any $k\in [s,r]$ and $\frakb_{r+1} \coloneqq \calJ_{K_m,\widetilde{G}}$. Then, $\frakb_k=\frakb_{k+1}\cap A_k$ and $\frakb_{k+1}+A_k=\calJ_{K_m,\widetilde{G}|_{V_k}}+A_k$ for all $k\in [s,r]$. 

    We have three cases.
    \begin{enumerate}[a]
        \item If $m\ge n$, then $\reg(S/\mathcal{J}_{K_m,{G}})=n-1$, by \Cref{lem:reg-min-equal}.

        \item Suppose that $n>m> n_r$.  In this case, we can see that for any $k\in [s,r]$, 
            \begin{gather*}
                \reg(S/\frakb_{r+1})=\reg(S/\calJ_{K_m,\widetilde{G}})=\min\{m-1,n-1\}=m-1, \\
                \reg(S/(\frakb_{k+1}+A_k))=\reg(S/(\calJ_{K_m,\widetilde{G}|_{V_k}}+A_k))=\min\{m-1,n_{k}-1\}=n_{k}-1,
                \intertext{and}
                \reg(S/A_k)=0,
            \end{gather*}
            as per \Cref{reg-m-n}. Consider the following  short exact sequence
            \begin{equation}
                0  \to S/\frakb_k  \to  S/\frakb_{k+1} \oplus S/A_k \to S/(\frakb_{k+1}+A_k)  \to 0
                \label{eqn:b_ideals}
            \end{equation}
            for such $k$. If we argue by induction on $k$ from $r+1$ down to $s$, it follows from \Cref{depthlemma}\ref{depthlemma-c} that $\reg(S/\frakb_k)=m-1$. In particular, $\reg(S/\mathcal{J}_{K_m,{G}})=\reg(S/\frakb_s)=m-1$.

        \item Suppose that $n_r \ge m\ge 2$. Likewise, we have 
            \begin{gather*}
                \reg(S/\frakb_{r+1})=\reg(S/\calJ_{K_m,\widetilde{G}})=\min\{m-1,n-1\}= m-1, \\
                \reg(S/(\frakb_{k+1}+A_k))=\reg(S/(\calJ_{K_m,\widetilde{G}|_{V_k}}+A_k))=\min\{m-1,n_{k}-1\}\le m-1,
                \intertext{and}
                \reg(S/A_k)=0
            \end{gather*}
            for each $k\in [s,r]$. Using a similar technique, we can show that $\reg(S/\frakb_k)\le m$ for $k=r+1,r,\dots,s$. In particular, $\reg(S/\mathcal{J}_{K_m,{G}})=\reg(S/\frakb_s)\le m$. 

            On the other hand, the complete $r$-partite graph $G$ contains an induced subgraph $K_{1,m}$, as $n_r\ge m$ and $r\ge 2$. Therefore, $\reg(S/\mathcal{J}_{K_m,{G}})\ge \reg(S/\calJ_{K_m,K_{1,m}})=m$ by \Cref{lem:induced-graph,star}. Thus, $\reg(S/\mathcal{J}_{K_m,{G}})=m$, which completes the proof. \qedhere
    \end{enumerate}
\end{proof}

\begin{Proposition}
    \label{prop:hilb}
    Under \Cref{r_setting}, the Hilbert series of $S/\mathcal{J}_{K_m,{G}}$ is given by 
    \begin{align*}
        \Hilb(S/\calJ_{K_m,G})= \frac{\sum_{i\ge 0}{m-1\choose i}{n-1\choose i}t^i}{(1-t)^{m+n-1}}+\sum_{k=s}^{r}\left(\frac{1}{(1-t)^{mn_k}} -\frac{\sum_{i\ge 0}{m-1\choose i}{n_k-1\choose i}t^i}{(1-t)^{m+n_k-1}}\right).
    \end{align*}     
\end{Proposition}
\begin{proof}
    We adopt the notation from the proof of \Cref{thm:reg_r_partite}. It follows from the short exact sequence in \eqref{eqn:b_ideals} that
    \[
        \Hilb(S/\frakb_k,t)= \Hilb(S/\frakb_{k+1},t)+ \Hilb(S/A_k,t) -\Hilb(S/(\calJ_{K_m,\widetilde{G}|_{V_k}}+A_k),t).
    \]
    Since $\calJ_{K_m,G}=\frakb_{s}$ and $\calJ_{K_m,\widetilde{G}}=\frakb_{r+1}$, we obtain that
    \begin{align*}
        \Hilb(S/\calJ_{K_m,G})= \Hilb(S/\calJ_{K_m,\widetilde{G}})+\sum_{k=s}^{r}\left(\Hilb(S/A_k,t) -\Hilb(S/(\calJ_{K_m,\widetilde{G}|_{V_k}}+A_k),t)\right).
    \end{align*}
    It follows from \cite[Corollary 1]{MR1213858} that
    \begin{align*}
        \Hilb(S/\mathcal{J}_{K_m,\widetilde{G}},t)&=\frac{\sum_{i\ge 0}{m-1\choose i}{n-1\choose i}t^i}{(1-t)^{m+n-1}} \\
        \intertext{and}
        \Hilb(S/(\mathcal{J}_{K_m,\widetilde{G}|_{V_k}}+A_k),t)&=\frac{\sum_{i\ge 0}{m-1\choose i}{n_k-1\choose i}t^i}{(1-t)^{m+n_k-1}}.
    \end{align*}
    Meanwhile, the Hilbert series of the polynomial rings $S/A_k$ is $\frac{1}{(1-t)^{mn_k}}$. Thus, the desired formula follows. 
\end{proof}

\begin{Corollary}
    Under \Cref{r_setting}, we assume additionally that $r=2$. Then
    the multiplicity of $S/\mathcal{J}_{K_m,{G}}$ is given by
    \[
        \e(S/\mathcal{J}_{K_m,{G}})=
        \begin{cases} 
            1, &\text{if $\max\{m+n_1+n_2-1,mn_1\}<mn_2$}, \\ 
            2, &\text{if $m+n_1+n_2-1<mn_1=mn_2$}, \\ 
            2n_2, &\text{if $mn_2<m+n_1+n_2-1$},  \\ 
            12, &\text{if $m+n_1+n_2-1=mn_1=mn_2$}, \\ 
            \sum_{k\ge 0}{m-1\choose k}{n_1+n_2-1\choose k}+1, &\text{if $mn_1< m+n_1+n_2-1=mn_2$}.  
        \end{cases}
    \]
\end{Corollary}

\begin{proof}
    It is a direct consequence of \Cref{prop:hilb}. 
    We only remark that if $mn_2<m+n_1+n_2-1$, then $m=2$ and $n_1=n_2$.
    Meanwhile, if $m+n_1+n_2-1=mn_1=mn_2$, then $m=3$ and $n_1=n_2=2$.
\end{proof}

In the following, we summarize the applications of the above results to binomial edge ideals of complete bipartite graphs. Specifically, we recover the corresponding parts of Corollary 3.4 and Theorem 5.4 (a) and (b) from Schenzel and Zafar's research \cite{MR3195706}.

\begin{Corollary}
    Let $G=K_{n_1,n_2}$ be a complete bipartite graph with $n_1\le n_2$. We have the following formulas for its binomial edge ideal $J_G$:
    \begin{enumerate}[a]
        \item  $\depth(S/J_G)=\begin{cases}
                n_2+2, &\text{\ if\ $n_1=1$,} \\
                n_1+2, & \text{\ if\ $n_1\ge 2$.}
            \end{cases}$
        \item The Hilbert series of $S/J_G$ is given by
            \[
                \Hilb(S/J_G,t)=\frac{1+(n_1+n_2-1)t}{(1-t)^{n_1+n_2+1}}+\frac{1}{(1-t)^{2n_1}}+\frac{1}{(1-t)^{2n_2}}
                -\frac{1+(n_1-1)t}{(1-t)^{n_1+1}}-\frac{1+(n_2-1)t}{(1-t)^{n_2+1}}.
            \]
        \item The multiplicity of $S/J_G$ is given by
            \[
                \e(S/J_G)=\begin{cases}
                    1, &\text{if $n_2>n_1+1$},\\ 
                    2n_2, &\text{otherwise.}
                \end{cases}
            \]
    \end{enumerate}
\end{Corollary}

We conclude this paper by showing that the binomial edge ideal $J_G$ of a complete $r$-partite graph $G$ is of  K\"onig type.
This builds on the work of Herzog et al.~\cite{MR4358668}, who introduced the concept of \emph{graded ideals of K\"onig type} with respect to a monomial order $<$, which generalizes the \emph{edge ideals of K\"onig graphs}. They proved that the Cohen--Macaulay property of the initial ideal $\ini_<(I)$ of $I$ is independent of the characteristic of the base field for any graded ideal $I$  of K\"onig type. 
Williams \cite[Theorem 2.4 and Proposition 2.24]{arXiv:2310.14410} recently demonstrated that the binomial edge ideal of a complete bipartite graph is of K\"onig type. In fact, this result also applies to complete $r$-partite graphs.

\begin{Proposition}
    If $G$ is a complete $r$-partite graph, the binomial edge ideal $J_G$ is of K\"onig type.
\end{Proposition}

\begin{proof}
    If $G$ is a complete graph, it is \emph{traceable}, meaning it has a spanning path as a subgraph. As a result, $J_G$ is of K\"onig type according to \cite[Proposition 3.6]{MR4358668}.

    In the following, we may assume that $G$ is a graph in \Cref{r_setting}. Let $h$ be the height of $J_G$. \Cref{r_dimension} states that $h=2n-\max\{n+1,2n_r\}$. To show that $J_G$ is of K\"onig type, it suffices to find a path of length $h$ in $G$  by \cite[Theorem 3.5]{MR4358668}. Let $n'\coloneqq n_1+n_2+\cdots+n_{r-1}$. There are two cases to consider:
    \begin{enumerate}[a]
        \item Suppose that $n_r>n'$. Notice that $V_1\sqcup \cdots \sqcup V_{r-1}=[n']$. The path we choose is:
            \[
                n'+1,1,n'+2,2,n'+3,3,\dots,n'+n',n',n'+(n'+1).
            \]

        \item Suppose that $n_r\le n'$. We can show that $G$ is {traceable}.
            For notational simplicity, we rewrite the vertices of $V_i$ as $V_i=\{v_{i,1},v_{i,2},\dots, v_{i,n_i}\}$ for $i\in [r]$. Thus, we obtain a total order $\succ$ on $V_1\sqcup \cdots \sqcup V_{r-1}$ by setting $v_{i,j}\succ v_{i',j'}$ if and only if $j<j'$, or $j=j'$ and $i<i'$. With respect to the order $\succ$, we can rewrite $V_1\sqcup \cdots \sqcup V_{r-1}$ as $\{u_1 \succ u_2 \succ \cdots \succ u_{n'}\}$. Furthermore, set $\delta\coloneqq n'-n_r+2$. The path  in $G$ we take will be:
            \[
                n'+1,u_1,u_2,\dots,u_\delta,
                n'+2,u_{\delta+1},n'+3,u_{\delta+2},\dots,n'+n_r-1,u_{\delta+n_r-2}=u_{n'},n'+n_r=n.
            \]
    \end{enumerate}
\end{proof}

\begin{Remark}
    When $G$ is a complete $r$-partite graph, the \emph{generalized} binomial edge ideal $\calJ_{K_m,G}$ is not necessarily of K\"onig type. For instance, the height of the ideal $\calJ_{K_3,K_{2,2}}$ is $6$, while the maximal length of the regular sequence in $\{\ini_{\lex}(p_{(e,f)}):e\in E(K_3),f\in E(K_{2,2})\}$ is $5$.
\end{Remark}

\medskip
\begin{acknowledgment*}
    This work is supported by the Natural Science Foundation of Jiangsu Province (No.~BK20221353). In addition,  the first author is partially supported by the Anhui Initiative in Quantum Information Technologies (No.~AHY150200) and the ``Innovation Program for Quantum Science and Technology'' (2021ZD0302902). And the second author is supported by the Foundation of the Priority Academic Program Development of Jiangsu Higher Education Institutions.  
\end{acknowledgment*}      

\bibliography{BEI} 
\end{document}